\documentclass[11pt]{article}
\usepackage{amsthm, amsmath, amssymb, amsfonts, url, booktabs, tikz, setspace, fancyhdr, bm}
\usepackage{hyperref}
\usepackage{geometry}
\usepackage{hyperref, enumerate}
\usepackage[shortlabels]{enumitem}
\usepackage[babel]{microtype}
\usepackage[english]{babel}
\usepackage[capitalise]{cleveref}
\usepackage{comment}
\usepackage{bbm,tkz-graph,subcaption}
\usepackage{csquotes}
\usepackage{mathrsfs}
\usepackage{mathabx}
\usetikzlibrary{patterns}
\usetikzlibrary{shapes}
\usepackage{bbm}

\usepackage[utf8]{inputenc} 
\usepackage[T1]{fontenc}
\usepackage{amsfonts}
\usepackage{amscd}
\usepackage{graphicx}
\usepackage{enumitem}
\usepackage{verbatim}
\usepackage{hyperref}
\usepackage{amsmath,caption}
\usepackage{url,pdfpages,xcolor,framed,color}
\usepackage{todonotes}
\usepackage{comment}

\newtheorem{thr}{Theorem}[section]

\newtheorem{lem}[thr]{Lemma}
\newtheorem{prop}[thr]{Proposition}
\newtheorem{conj}[thr]{Conjecture}
\theoremstyle{definition}

\newtheorem*{defi*}{Definition}
\newtheorem{cor}[thr]{Corollary}

\newtheorem{remark}[thr]{Remark}
\newtheorem{examp}[thr]{Example}

\newtheorem{claim}[thr]{Claim}

\def\Z{\mathbb{Z}}
\def\res{\mathrm{res}}

\def\eps{\varepsilon}

\newcommand*{\floorfrac}[2]{\mathopen{}\left\lfloor\frac{#1}{#2}\right\rfloor\mathclose{}}
\newcommand*{\abs}[1]{\left\lvert #1\right\rvert}

\def\a{\alpha}
\def\b{\beta}

\newenvironment{poc}{\begin{proof}[Proof of claim]}{\end{proof}}

\newcommand{\di}{i} 

\newcommand*{\myproofname}{Proof}

\title{The Erd\H{o}s distinct subset sums problem in a modular setting}

\date{}
\author{
Stijn Cambie \thanks{Extremal Combinatorics and Probability Group (ECOPRO), Institute for Basic Science (IBS), Daejeon, South Korea, supported by the Institute for Basic Science (IBS-R029-C4),
E-mail: {\tt stijn.cambie@hotmail.com,\{jungao, hongliu, younjinkim\}@ibs.re.kr}.} \thanks{Department of Computer Science, KU Leuven Campus Kulak-Kortrijk, 8500 Kortrijk, Belgium. Supported by Internal Funds of KU Leuven (PDM fellowship PDMT1/22/005).}
\and
Jun Gao\footnotemark[1]
\and
Younjin Kim\footnotemark[1] 
\and
Hong Liu\footnotemark[1] 
}

\begin{document}
\maketitle
\begin{abstract}
    We prove the following variant of the Erd\H{o}s distinct subset sums problem. 
    Given $t \ge 0$ and sufficiently large $n$, every $n$-element set $A$ whose subset sums are distinct modulo $N=2^n+t$ satisfies 
    $$\max A \ge \Big(\frac{1}{3}-o(1)\Big)N.$$ 
    Furthermore, we provide examples showing that the constant $\frac 13$ is best possible. For small values of $t$, we characterise the structure of all sumset-distinct sets modulo $N=2^n+t$ of cardinality $n$.
\end{abstract}

\section{Introduction}

The Erd\H{o}s distinct subset sums problem, considered one of the oldest conjectures of Erd\H{o}s~\cite{Erdos89}, poses the question whether there exists a fixed constant $c>0$ such that for every $n$-element set $A=\{a_1,a_2,\cdots, a_n\}$ of positive integers, all of whose subsets have distinct sums, the largest element, say $a_n=\max A$, has to be greater than $c \cdot 2^n.$ 
Erd\H{o}s and Moser~\cite{Erdos56} first observed that $a_n\ge \frac{2^n}{n}$ and later improved this lower bound to $\Omega\left( \frac{2^n}{\sqrt n} \right)$ by considering the sum of the squares of the elements~\cite{Guy82}. Erd\H{o}s has posed this question multiple times and even offered monetary prize for its proof or disproof~\cite{Erdos70}.
The lower bound on the largest element $a_n$ has subsequently been improved by a constant factor 
~\cite{As77,Guy82,Elkies86,Bae95,Aliev08,DFX21}.
The current best known lower bound $\left(\sqrt{\frac{2}{\pi}}-o(1)\right)\frac{2^n}{\sqrt n}$ is due to Dubroff, Fox and Xu~\cite{DFX21}; an alternative proof is given by Steinerberger~\cite{Sb22}.
In the other direction, starting with a construction by Conway and Guy~\cite{CG68}, where they produced a construction with $a_n \leq 2^{n-2}$, better constructions have been obtained 
~\cite{Lunnon88,ANS90,Maltby97}.
The best construction so far is the one by Bohman~\cite{Bohman97}, where he produced a construction with $a_n \leq 0.22002\cdot 2^n.$
Here he studied variations of the Conway-Guy sequence. These sequences are derived from differences between an element and the smaller ones in an increasing sequence $(u_n)_n$ which is exponential: $S=\{u_n-u_1, u_n-u_2, \ldots, u_n-u_{n-1} \}.$
As such, most of the elements in the set have roughly the same magnitude.
This implies that two subsets of different sizes have a different subset-sum in most cases.
For two subsets with the same size, it is sufficient to have a different subset-sum modulo $u_n$.

Motivated by this, we consider the Erd\H{o}s distinct subset sums problem in a modular setting.
An  $n$-element set $\{a_1,a_2,\cdots, a_n\}$ of positive integers is said to  be \emph{sumset-distinct modulo $N$} if all subsets (including the empty set) have distinct sums modulo $N.$
Note that we may assume $N\ge 2^n$ for otherwise not all subset sums of a size-$n$ set can be distinct modulo $N$. We propose the following modular version of the Erd\H{o}s distinct subset sums conjecture.

\begin{conj}\label{conj:EDSC_modular}
    For a fixed $t\ge 0$ and sufficiently large $n$, let $A$ be an $n$-element set of positive integers, all of whose subsets have distinct sums modulo $N=2^n+t$, it holds that $$\max{A} \ge N/3\ge 2^n/3.$$
\end{conj}

The condition that $n$ is sufficiently large is necessary, as shown by the following example. Let $n=9, t=104$ and consider the set $\{77,117,137,148,154,157,159,160,161\}$. It is sumset-distinct modulo $616=2^9+104$, but  $161<\frac{1}3 2^9 < \frac N3$.

In this paper, we prove the following asymptotic version of Conjecture~\ref{conj:EDSC_modular}. Comparing the constant $\frac{1}{3}$ below with the one ($0.22002$) obtained by Bohman~\cite{Bohman97}, our result witnesses a slightly different behaviour of sumset-distinct sets in the modular setting.

\begin{thr}\label{thr:EDSC_modular}
    Let $0<\eps<1$ and $n,t$ be positive integers such that $t<(1+ \eps)^n$. 
    Assume that $A=\{a_1,a_2,\cdots,a_n\}$ is an $n$-element set of positive integers in which the sums of all its subsets are distinct modulo $N=2^n+t.$ Then 
    $$\max{A} \ge (1-\eps)N/3.$$
    Furthermore, the constant $1/3$ is optimal. That is, There exist sets that are sumset-distinct modulo $N=(1+o(1))2^n$ for which $\max A\le(\frac{1}{3}+o(1))N.$
\end{thr}

A nice consequence of~\cref{conj:EDSC_modular} is that it can be used to prove the original Erd\H{o}s distinct subset sums problem when the sum of all elements $\sum_{i=1}^na_i$ is close to $2^n$.

\begin{prop}\label{prop:application}
    If \cref{conj:EDSC_modular} is true, then for any $A=\{a_1, a_2, \ldots, a_n\}$, $a_1< a_2<\ldots <a_n$, with distinct subset sums and $\sum_{i=1}^n a_i<2^n+t$, we have $2^{i-1} \le a_i \le 2^{i-1}+t$ for every large $i$.
\end{prop}

Lending support to Conjecture~\ref{conj:EDSC_modular}, we establish its validity for $0\le t \le 3$ in strong sense by characterising the structure of all maximum-size sumset-distinct sets modulo $N=2^n+t$. 

We show that a set of cardinality $n$ is sumset-distinct modulo $2^n$ if and only if each element has a distinct $2$-adic valuation. The $p$-adic valuation of an integer $N$ is the integer $k$ such that $p^k$ divides $N$, but $p^{k+1}$ does not. This valuation is denoted as $v_p(N).$

\begin{thr}\label{thm:0}
    A set $A$ of cardinality $n$ is sumset-distinct modulo $N=2^n$ if and only if its elements can be ordered as $A=\{a_0,a_1, \ldots, a_{n-1}\}$ such that $v_2(a_j)=j$ for every $0 \le j \le n-1$. 
\end{thr}

When $N=2^n+1$ or $N=2^n+3$, we prove that maximum-size sumset-distinct sets modulo $N$, up to signs, are dilations of the powers of $2.$

\begin{thr}\label{thm:1-3}
     A set $A$ of cardinality $n$ is sumset-distinct modulo $N=2^n+1$ or $2^n+3$ for some integer $n\geq 1$ if and only if 
      it can be expressed as  $ A=\{ \pm \lambda 2^0, \pm \lambda 2^1, \ldots,  \pm \lambda 2^{n-1}\} \pmod N$ for some $\lambda$ satisfying $\gcd(\lambda, N)=1$ and some choice of the signs.
\end{thr}

\begin{cor}\label{cor:1}
    For every integer $n\geq 1$, there are exactly $\frac{2^n \phi{(N)}}{2n}$ sumset-distinct sets modulo $N=2^n+1$ of cardinality $n$.
\end{cor}

Modulo $N=2^n+2$ is closely related to modulo $2^{n-1}+1$.

\begin{thr}\label{thm:2}
     A set $A$ of cardinality $n$ is sumset-distinct modulo $N=2^n+2$ for some integer $n \ge 2$ if and only if\\ 
    (i) $A$ is the union of an odd number and $2A'$, where $A'$ is sumset-distinct modulo $N/2,$ or\\
 (ii) $A$ is the union of $N/2$ and $A'$, where $A'$ is sumset-distinct modulo $N/2.$ 
    In particular, the number of such sets is $O(N^3).$
\end{thr}

 Somewhat surprisingly, our results show that there are more solutions for $N=2^n$ (when every sum appears exactly once) compared to $N=2^n+t$ for any $t>0$ and $n$ sufficiently large.
In the former case, the number of solutions is superpolynomial, at least $2^{\Omega(n^2)}=2^{\Omega(\log^2 N)}$, whereas in the latter case, for $t$ odd, the number of solutions is polynomial, bounded by $O\left(\sqrt{N}^{3+t}\right)$.

\section{Preliminaries}\label{subsec:background}

For a given set $A$, we define $s(A)=  \sum_{a \in A} a$ as the sum of its elements, and $S(A)$ as the set of sums of subsets of $A$, that is, $S(A)= \left\{s(B) \colon B \subseteq A\right\}.$ A translate of a set is denoted by $A+b=\{a+b:~a\in A\}$. A primitive $k$th root of unity is a solution to the equation $x^k=1$ that does not satisfy $x^m=1$ for any $m$ dividing $k$.
We use $\omega_k$ for the standard $k$th root of unity, which is equal to $\exp\left( \frac{2 \pi \di }{k} \right)$, so $1+\omega_k+\omega_k^2+\ldots + \omega_k^{k-1}=0$. For any $x \in \mathbb Z$, let $\res(x)=\min\{ x \pmod N, N- x\pmod N \}$ be the element between $0$ and $\floorfrac N2$ which is equivalent to $x$ or $-x$ modulo $N.$

We start with two lemmas which indicate that there can be multiple sumset-distinct sets modulo $N$ which are essentially the same. This implies that one can partition the sumset-distinct sets modulo $N$ into equivalent classes of sets which are related with a composition of two operations: dilation or changing signs. 

\begin{lem}\label{lem:obs1}
    A set $A=\{a_1, \ldots, a_n\}$ is sumset-distinct modulo $N$ if and only if each of the $2^n$ sets of the form $\{\pm a_1, \ldots, \pm a_n\}$ are sumset-distinct modulo $N$.
\end{lem}

\begin{proof}
    It suffices to observe that if $\{a_1, a_2 \ldots, a_n\}$ is a sumset-distinct set modulo $N$, then $\{-a_1, a_2, \ldots, a_n\}$ is also a sumset-distinct set modulo $N$. This can be easily seen: if there exist disjoint sets $I_1$ and $ I_2$ from the set $[n]$ such that $\sum_{i \in I_1 } a_i \equiv \sum_{i \in I_2 } a_i \pmod N$, then the same sets work when $1 \not \in I_1 \cup I_2$. Otherwise, if $1$ is included in either $I_1$ or $I_2$, we can simply switch the element $1$ between the two sets.
    In other words, let $I'_1=I_1 \Delta \{1\}$\footnote{Here $\Delta$ denotes the symmetric difference of the two sets.} and $I'_2=I_2 \Delta \{1\}$. These are two distinct index subsets of $[n]$, corresponding to two subsets of elements in $\{-a_1, a_2, \ldots, a_n\}$ with the same sum modulo $N.$
\end{proof}

\begin{lem}\label{lem:obs2}
For a given $\lambda$ where $\gcd(\lambda, N)=1,$
    a set $A=\{a_1, \ldots, a_n\}$ is sumset-distinct modulo $N$ if and only if $\{\lambda a_1, \ldots, \lambda a_n\}$ is also sumset-distinct modulo $N$.
    \end{lem}

\begin{proof}
    It suffices to note that $\sum_{i \in I_1 } a_i \equiv \sum_{i \in I_2 } a_i \pmod N$ if and only if $\lambda \sum_{i \in I_1 } a_i \equiv \lambda \sum_{i \in I_2 } a_i \pmod N$.
\end{proof}

We will also need the following corollary of~\cite[Thr.~1.5]{GK23}.
Here we use that the group $\left(\frac{\mathbb Z}{N \mathbb Z},+\right)$ is abelian and does not have elements of order $2$ when $N$ is odd.

\begin{thr}[\cite{GK23}]\label{thr:GK23}
    Let $A,B$ be multisets in $G= \frac{ \mathbb Z}{N \mathbb Z}$ for an odd positive integer $N.$
Then $S(A)=S(B)$ implies that $A$ be obtained from $B$ by applying one of the following moves
\begin{enumerate}
    \item replace an element $x$ by $-x$;
    \item for some $\a, \b$ relative prime with $N$, replace the elements in the set $\a U_k$ by $\b U_k$ where $U_k=\{1,2,\ldots, 2^{k-1}\}$ and $k>0$ is the smallest integer for which $2^k \equiv \pm 1 \pmod N$.
\end{enumerate}
    When the cardinality of $A$ and $B$ is $n$, we have that
    \begin{itemize}
        \item when $N=2^{n}+1,$ $S(A)=S(B)$ implies that the elements of $B$ can be obtained up to sign from multiplying all elements in $A$ with some element $\lambda =\b \a^{-1}$ which is relative prime to $N,$
        \item when $N>2^n+1,$ $S(A)=S(B)$ implies that the elements in $A$ and $B$ are equal up to sign.
    \end{itemize}
\end{thr}

For $N\ge 2^n+3$, it is important to note that the smallest positive integer $k\ge 1$ satisfying $2^k \equiv \pm 1 \mod N$ must satisfy $k \ge n+1$ and thus no multiple of $U_k$ can be a subset of a set with cardinality $n.$

\section{The asymptotic modular Erd\H{o}s distinct subset sums conjecture}\label{sec:EDSC_modular}

\begin{proof}[Proof of Theorem~\ref{thr:EDSC_modular}]
    Let us define the sumset generating function 
    $$g_A(x)=\prod_i (1+x^{a_i})$$ 
    and let $\{s_1, s_2, \ldots, s_t\}=\frac{\Z}{N\Z}\setminus S(A)$ be the set of $t$ numbers that cannot be expressed as the sum of a subset of $A$. Then we have
    $$g_A(\omega_N)+\sum_{i=1}^t \omega_N^{s_i}= \sum_{i=0}^{N-1}\omega_N^i=0.$$
  Hence, by the triangle inequality, we have
    $$t\ge \abs{\sum_{i=1}^t \omega_N^{s_i}}=\abs{g_A(\omega_N)} = \prod_{i=1}^{n} \abs{1+\omega_N^{a_i}}
= \prod_{i=1}^{n} \sqrt{2+2\cos \left(\frac{a_i}{N} \cdot 2\pi  \right)}.$$
Now, if $\max{A} < (1-\eps) \frac N3$, then using $\cos\left( \frac{\pi}{3} + \frac{2\pi}{3} x\right) < \frac 12 -x-\frac{x^2}{2}$ for $0<x<1$, we have for every $i\in [n]$ that
\begin{align*}
  \sqrt{2+2\cos \left(\frac{a_i}{N} \cdot 2\pi  \right)} &> \sqrt{2+2\cos \left( \frac{1-\eps}{3} \cdot 2\pi  \right)}
  = \sqrt{2-2 \cos\left( \frac{\pi}{3}+\frac{2\eps\pi}{3}\right)}>1+\eps.
\end{align*}
 This would imply $t>(1+\eps)^n$, contradicting the assumption.

\smallskip

We now construct the sets showing that the constant $c=\frac{1}{3}$ cannot be improved.

Let $n\gg k\gg 1$ be two integers, and consider $N=2^n+2^k+1$ such that $3 \mid (N-1)$, and so $3\mid 2^{n-k}+1$.
    Let $B = \{a_1,a_2,\cdots,a_n\}$, where
    \begin{align*}      
        a_i = \left \{
        \begin{aligned}
             &2^{i-1},  &1\le i\le n-k,\\
             &2^{i-1-(n-k)}(2^{n-k}+1), &n-k< i\le n. 
        \end{aligned}
        \right.
     \end{align*} 
     Since $\sum_{i=1}^j a_i \le a_{j+1}$ and $\sum_{i=1}^n a_i<N$, $B$ is  sumset-distinct module $N$.
     Let $\lambda = \frac{N-1}{3}$, then $\gcd(\lambda,N)=\gcd(N-1,N)=1$. By Lemma~\ref{lem:obs2}, we know that $\lambda B$ is also sumset-distinct modulo $N$.
     Let $A$ be the set obtained from $\lambda B$ by restricting all elements to $[N/2]$, that is, replacing $x$ by $\res(x)=\min\{ x \pmod N, N- x\pmod N \}$.
     Since $\lambda a_i = \frac{a_i}{3}N-\frac{a_i}{3}$, it is not hard to see that for $i\le n-k$, either $\lambda a_i<\frac{N}{3} + \frac{N}{2^k} \pmod N$ or $\lambda a_i>\frac{2N}{3} - \frac{N}{2^k} \pmod N$.
     If $i>n-k$, we have $3 \mid N-1$ and thus $\lambda a_i \equiv -\frac{a_i}{3} \pmod N$.
     Thus, we conclude that $\max A \le \frac{N}{3} + \frac{N}{2^k}=(\frac{1}{3}+o(1))N$ as $n\gg k\gg 1$.     
\end{proof}

We complement~\cref{thr:EDSC_modular} with the following observation for the original setting.

\begin{prop}\label{prop:no_el_inN/3..}
    Let $A=\{a_1, a_2, \ldots, a_n\} \in \binom{\mathbb Z^+}{n}$ be a set in which the sums of its subsets are distinct.
    Let $N=1+\sum_{i=1}^n a_i $.
    If $a_i\le 2^{n-1}-1$, then $a_i < \frac{N}{3}.$
\end{prop}

\begin{proof}
    Let us consider an element $a_i$ that is at most $2^{n-1}-1.$
    Note that the interval $[0,2a_i-1]$ can contain at most $a_i<2^{n-1}$ elements from the set $S(A \setminus \{a_i\})$. Consequently, $S(A \setminus \{a_i\})$ contains an element that is at least $2a_i$. This implies that $S(A)$ contains an element that is at least $3a_i$, and therefore $3a_i<N.$
\end{proof}

We can now prove Proposition~\ref{prop:application} using~\cref{prop:no_el_inN/3..}.

\begin{proof}[Proof of Proposition~\ref{prop:application}]
Let us consider  $A=\{a_1, a_2, \ldots, a_n\}$ with $a_1< a_2<\ldots <a_n$ having distinct subset sums and satisfying $\sum_{i=1}^n a_i<2^n+t$ for a fixed value of $t$. Then $A$ has also distinct sums modulo $2^n+t$. Assuming \cref{conj:EDSC_modular}, we have $\max{A} \ge \frac {2^n+t}{3} \ge \frac{1+\sum_{i=1}^n a_i}{3}.$ Applying~\cref{prop:no_el_inN/3..}, we conclude that $a_n \ge 2^{n-1}$.
Since $A$ has distinct subset sums, $\sum_{i=1}^{n-1} a_i \ge 2^{n-1}-1$. Consequently, $a_n \le 2^{n-1}+t.$
Now one can repeat the argument for the set $\{a_1, a_2, \ldots, a_{n-1}\}.$
\end{proof}

\section{Sumset-distinct sets modulo $N=2^n$}\label{sec:2^n}

In this section, we prove Theorem~\ref{thm:0}. The ``if'' direction can be easily proven by induction.
Since the cardinality of ${S(A)}$ is $2^n$, it suffices to observe that for every integer $k$, there is a solution for
$\eps_0 a_0 +\eps_1 a_1+ \ldots +\eps_{n-1}a_{n-1} \equiv k \pmod {2^n}$,
where $\eps_i \in \{0,1\}$ for every $0 \le i \le n-1.$
Note that $\eps_0 \equiv k \pmod 2.$ 
By the induction hypothesis for $n'=n-1$ and $a_j'=a_{j+1}/2$, $0\le j\le n'-1$, we find a solution for $\eps_1 a'_0+\eps_2 a'_1+ \ldots +\eps_{n-1}a'_{n'-1}\equiv \frac{k-\eps_0 a_0}{2} \pmod {2^{n-1}}$, implying that $\eps_1 a_1+ \ldots +\eps_{n-1}a_{n-1}\equiv k-\eps_0 a_0 \pmod {2^n}$, as desired. 

We provide two proofs of the ``only if'' direction. For each $0 \le j\le n-1,$ let $A_j =\{a \in A \mid v_2(a) \le j\}.$ 

\begin{proof}[Proof 1]
    Consider the generating function 
    $G_j(x)= \prod_{a \in A_j} \left( 1+x^a \right),$ which, in its expanded form, equals $\sum_{ a \in S(A_j)} x^a$.
    Since $A$ is sumset-distinct modulo $2^{n},$ every residue modulo $2^{j+1}$ appears equally often among the elements in $S(A_j).$
    By considering the expansion as a sum and using $\sum_{i=0}^{N-1} \omega_{2^{j+1}}^i=0$, we can deduce that $G_j(x)$ is zero when evaluated in 
    $\omega_{2^{j+1}}= \exp\left( \frac{ \pi \di}{2^j} \right).$
    The factorization of $G_j(x)$ implies the existence of a factor $1+\omega_{2^{j+1}}^a =0,$ indicating that there is at least one value $a$ for which $a \equiv 2^{j} \pmod{2^{j+1}},$ or equivalently, $v_2(a)=j.$
    Since this holds for every $0 \le j \le n-1$ and considering that $\abs{A}=n,$ 
    we can conclude that the elements of $A$ can be ordered such that $v_2(a_j)=j$, establishing the desired conclusion.
\end{proof}

\begin{proof}[Proof 2]
Assume, for the sake of contradiction, that there exists a set $A$ of cardinality $n$ that is sumset-distinct modulo $2^n$ but does not have the claimed form.
Let us consider such a set $A$ for which the cardinality $n$ is minimal.
We observe that $A$ must contain an odd number, as $S(A)$ contains all $2^n$ residues modulo $2^n$ and therefore includes the odd residues as well.
If $A_0=\{a_0\}$ would be a singleton, then $(A\setminus A_0)/2=\left\{ a/2 \mid a \in A\setminus A_0 \right \}$ would be sumset-distinct modulo $2^{n-1}$ and would not have the claimed form either.
This would contradict our choice of $A$ being a minimal example. Thus, we conclude that $\abs{A_0}\ge 2.$

Since $A$ is sumset-distinct modulo $2^n$, $S(A)=\{0,1,2,3\cdots,2^{n}-1\} \pmod{2^n}$ and so
\begin{align*}
    2^{n-1}\sum_{a\in A} a=\sum_{B\subseteq A} \sum_{a\in B} a\equiv \sum_{i=1}^{2^{n}-1} i
    =2^{n-1}(2^n-1) \pmod{2^n},
\end{align*}
which implies that $\sum_{a\in A} a$ is odd.

Let $N_e = \left\{B \colon \sum_{a \in B} a \text{ is even},  B \subseteq A\right\}$, and let $t:=|A_0|\ge 2$.

\begin{claim}\label{clm:sum_evensums}
$\sum_{B\in N_e} \sum_{a\in B}a =2^{n-2}\sum_{a\in A}a$
\end{claim}
\begin{poc}
    When an odd element in $A$ is fixed, the number of choices for the other elements (to form $B$) such that the sum is even is given by $$2^{n-t}\cdot \sum_{0\le i\le t-1,i \text{ is odd}} \binom{t-1}{i} =2^{n-t+t-2}=2^{n-2},$$ where we have used the fact that $t\ge 2$. 

    Similarly, when an even element in $A$ is fixed, the number of choices for the other elements such that the sum is even is $$2^{n-t-1}\cdot \sum_{0\le i\le t,i \text{ is even}} \binom{t}{i} =2^{n-t-1+t-1}=2^{n-2}.$$
\end{poc}

Using Claim~\ref{clm:sum_evensums}, we can now derive that:
\begin{align*}
    2^{n-2}\sum_{a\in A} a= \sum_{B\in N_e} \sum_{a\in B}a \equiv \sum_{i=1}^{2^{n-1}-1}2i= 2^{n-1}\cdot(2^{n-1}-1) \pmod{2^{n}}.
\end{align*}
However, this implies that $\sum_{a\in A} a$ is even, a contradiction.
\end{proof}

\section{Sumset-distinct sets modulo $N=2^n+1$ or $N=2^n+3$}\label{sec:2^n+1}

In this section, we prove Theorem~\ref{thm:1-3}. The ``if'' direction is trivial, following from applying \cref{lem:obs1} and~\cref{lem:obs2} on $A=\{1,2,2^2,\ldots, 2^{n-1}\}.$

For the ``only if'' direction, we begin by proving a claim that will be used for both cases.

\begin{lem}\label{lem:relativeprime}
        If $N$ is $2^n+1$ or $2^n+3$, then every element in a sumset-distinct set (modulo $N$) $A$ must be relatively prime to $N.$
    \end{lem}

    \begin{proof}
        Suppose there exists an element $a \in A$ and a common divisor $d$ such that $d \mid \gcd(a, N)$ and $d>1.$ Note that if $N=2^n+3$, then $d>3.$
        Since $S(A)$ contains exactly the elements in $S(A \setminus \{a\})$ along with its translate $a+S(A \setminus \{a\})$, every residue modulo $d$ appears an even number of times.
        Now consider the fact that $S(A)$ contains all residues modulo $N$, except for  exactly one residue (when $N=2^n+1$) or at most three residues (when $N=2^n+3$). Hence, at least one element (at least $d-1$ in the case of $N=2^n+1$, or at least $d-3$ in the case of $N=2^n+3$ ) appears $\frac{N}{d}$ times modulo $d.$ 
        However, since $N$ is odd, we know that $\frac Nd$ is also odd. This leads to a contradiction with the previous observation that every residue appears an even number of times modulo $d$.
    \end{proof}

We first prove the ``only if'' direction for $N=2^n+1$.

\begin{proof}[Proof for $N=2^n+1$]
   Let us fix one element $\lambda \in A$ and consider $A':=\lambda^{-1}A=\{\lambda^{-1}a\mid a \in A\},$ 
    where $\lambda^{-1}$
    is the inverse of $\lambda$ in $\frac{\mathbb Z}{N \mathbb Z}.$
    Since $\gcd(\lambda,N)=1,$ we have that $\gcd(\lambda^{-1},N)=1,$ and so $A'$ is also sumset-distinct modulo $N$ by~\cref{lem:obs2}.
    We note that by definition, $1$ belongs to  $A'$ and there is exactly one residue $x$ modulo $N$ that is not part of $S(A').$
    Furthermore, it is evident that for any $y$, if $y, y+1 \not \equiv x \pmod N,$ then exactly one of the two belongs to $S(A' \setminus \{1\}).$
    
    By writing the residues as the elements $x+1-N$ up to $x-1$, we can deduce that $S(A' \setminus \{1\}) \pmod N$ contains exactly the even numbers within the interval from $x+1-N$ up to $x-1$.
    Consequently, we can express $A'=\{1,a_2, a_3, \ldots, a_n\}$, where each $a_i$ is an even number within the interval $[x+1-N, x-1]$.
    We now present  three elementary observations:
    \begin{itemize}
        \item the sum of any set of negative numbers in $A' \setminus \{1\}=\{a_2, a_3, \ldots, a_n\}$ is an even negative number,
        \item since $N$ is odd, $S(A' \setminus \{1\})$ is not congruent modulo $N$ to any even number between $x-2N$ and $x-N$,
        \item for every $2 \le i \le n$, we have $\abs{a_i}< N$.
    \end{itemize} 
    Combining these observations, we can conclude that $S(A' \setminus \{1\})$ does not contain any element smaller than $x+1-N.$ Similarly, no element in $S(A' \setminus \{1\})$ is larger than $x-1.$
    Consequently, $S(A')$ is precisely equal to the interval $[x+1-N, x-1]$. A corollary of this is that the sum of all negative numbers in $A'$ results in the smallest number, $x+1-N$, and the sum of all positive numbers in $A'$ equals $x-1$.
    By replacing the negative numbers in $A'$ with their additive inverse, we obtain $A''$, which is still sumset-distinct modulo $N$ due to~\cref{lem:obs1}. Additionally, the sum of all elements in $A''$ equals $N-(x+1)+(x-1)=N-2=2^n-1$ and so $S(A'')=\{0,1,2,\ldots, 2^n-1\}.$ It remains to show that $A''=\{1,2,4,\ldots, 2^{n-1}\}.$ Suppose we know the $i$ smallest elements of $A''$ are $1,2,4,\ldots, 2^{i-1}$, for some $i$ with $0\le i\le n-1$, then since $2^i$ is the smallest element in $S(A)$ that cannot be expressed as a sum of any subset of these $i$ smallest elements, it must belong to $A''$. By induction, we can establish the desired conclusion.
\end{proof}

\begin{proof}[Proof of~\cref{cor:1}]
    Recall that $\res(x)=\min\{ x \pmod N, N- x\pmod N \}$ is the element between $0$ and $\floorfrac N2$ equivalent to $x$ or $-x$ modulo $N.$
    For every $\lambda \in [N]$ for which $\gcd(\lambda , N)=1,$ we consider 
    $R( \lambda)=\{ \res(\lambda), \res(2\lambda), \ldots, \res(2^{n-1}\lambda) \}$ as a representative set for all the sets of the form $\{ \pm \lambda, \pm 2 \lambda, \ldots \pm 2^{n-1}\lambda\}.$
    We note that $R(2\lambda)=R(\lambda)$ since $\res(2^n \lambda) = \res(-\lambda)$, as a consequence of $2^n \equiv -1 \pmod N$.
    This implies that $R(\lambda)=R( \pm 2^i \lambda)$ for every $0\le i \le n-1.$
    On the other hand, if $\lambda' \not \in \{ \pm \lambda, \pm 2 \lambda, \ldots \pm 2^{n-1}\lambda\}$, as elements in $\frac{ \mathbb Z}{N \mathbb Z}$, then $R(\lambda) \not= R(\lambda')$ and furthermore these sets are disjoint.
    There are $\phi(N)$ choices for $\lambda$, and thus $\frac{\phi(N)}{2n}$ different representative sets $R(\lambda).$
    Every set $R(\lambda)$ is associated with $2^n$ different sets of the form $\{ \pm \lambda, \pm 2 \lambda, \ldots \pm 2^{n-1}\lambda\}.$
    We conclude that there are exactly $\frac{2^n \phi{(N)}}{2n}$ sumset-distinct sets modulo $N=2^n+1.$
\end{proof}

Finally, we prove the case $N=2^n+3$ and begin by presenting a lemma.

\begin{lem}\label{lem:3APs}
    Suppose $A$ is a sumset-distinct set modulo $N=2^n+t$, where $t$ is an odd positive integer, and $A$ has cardinality $n$.
    In this case, the missing elements from $S(A)$ can be divided into a singleton $x$ and $\frac{t-1}{2}$ pairs $\{y,z\}$ such that $y+z=2x.$
\end{lem}

\begin{proof}
    Let $T:=s(A)$ be the sum of all elements in $A$.
    For every subset $B\subset A,$ we have $s(A \setminus B)+s(B)=T.$
    Therefore, an element $a$ belongs to $S(A)$ if and only if $T-a$ does.
    We note that no subset $B$ can satisfy $s(B)=\frac{T}2$, as otherwise $A \setminus B$ has the same sum.
    Since $[N]$ can be partitioned into $\frac{N-1}{2}$ pairs $\{a,b\}$ with $a+b=T$, along with the element $\frac{T}2,$ it follows that the element $\frac{T}2$ and $\frac{t-1}{2}$ pairs are missing, thereby establishing the lemma.
\end{proof}

\begin{proof}[Proof for $N=2^n+3$]
    According to~\cref{lem:3APs}, the three missing numbers in the sumset-distinct set modulo $N$ form an arithmetic progression of length $3.$ 
    We first argue that the common difference of these three numbers is relatively prime to $N.$

    \begin{claim}\label{clm:common_diff_relprime}
    Let $A$ be a sumset-distinct set modulo $N=2^n+3$, and let $x,y,z$ be the three missing elements in $S(A)$ such that $y+z=2x$. Then we have $\gcd(x-y,N)=1.$
    \end{claim}
    \begin{poc}
        Suppose to the contrary that there is a prime divisor $p \mid \gcd(x-y,N)$.
        Since $N=2^n+3$ is not divisible by $2$ or $3$, we have $p>3$.
        Furthermore, all three elements  $x,y$, and $z$ are congruent modulo $p.$ Thus, among the residue classes modulo $p$, $S(A)$ contains $\frac{N}{p}-3$ elements from the one containing $x$, and exactly $\frac{N}{p}$ elements from every other residue class.
        Let $A=\{a_1,a_2, \ldots, a_n\}$ and $A_i=\{a_1,a_2, \ldots, a_i\}.$
        Then we have $S(A_{i+1})=S(A_i) \cup \left( S(A_i) + a_{i+1}\right).$
        Let the number of elements in $S(A_i)$ that are congruent to $j$ modulo $p$ be denoted as $x_{i,j}.$
        By definition, we have $x_{n,j} \equiv \frac{N}{p} \pmod 3$ for every $j \in [p]$.
        Using induction, we shall prove that for a fixed $i \in [n],$ $x_{i,j} \pmod 3$ is the same for every $j \in [p]$ and not equal to zero. This suffices to conclude the proof as some of the $x_{1,j}$ values are clearly equal to zero, a contradiction.

        Assuming the statement is true for $i+1$ to $n$, we will prove that it is also true for $i$. Note first that as $S(A_{i+1})=S(A_i) \cup \left( S(A_i) + a_{i+1}\right)$, by definition, we have
    $$x_{i+1,j}=x_{i,j}+x_{i,j-a_{i+1}},$$    
        where the index $j-a_{i+1}$ is considered modulo $p$ as well.
        It is important to note that, by~\cref{lem:relativeprime}, $p \nmid a_{i+1}.$ So we can consider these equations with variables $x_{i,k \cdot a_{i+1}}$ for $k \in [p]$ and constants $x_{i+1,j}$, 
        and rewrite the system of linear equations as follows:

        \[
\begin{bmatrix}
1 & 1 & 0 &0& \dots & 0 &0 \\
0 &1& 1 &0& \dots & 0 &0\\
\dots  & \dots  & \ddots  & \ddots & & \vdots   \\
0 & 0 & 0 & 0& \dots & 1&1 \\
1 & 0 & 0 & 0& \dots & 0&1 \\
\end{bmatrix}
\begin{bmatrix}
x_{i, a_{i+1}} \\ x_{i, 2a_{i+1}} \\  x_{i, 3a_{i+1}}  \\ \dots\\ x_{i, -a_{i+1}}   \\ x_{i,0}
\end{bmatrix}
=
\begin{bmatrix}
 x_{i+1, 2a_{i+1}} \\  x_{i+1, 3a_{i+1}}  \\ \dots\\ x_{i+1, -a_{i+1}}   \\ x_{i+1,0} \\x_{i+1, a_{i+1}} 
\end{bmatrix}
\]

The coefficient matrix in this case is a $p \times p$-matrix.
Since $p$ is odd, the determinant of the matrix is $2.$
Consequently, in $\frac{\mathbb Z}{3\mathbb Z},$ there exists a unique solution.
However, since all the constants $x_{i+1,j}$, $j\in[p]$, are the same by the induction hypothesis, this implies that all the variables $x_{i,j}$ are the same and equal to $2x_{i+1,j}$ modulo 3.
\end{poc}

    Let $\Delta$ be the common difference of $x,y,$ and $z$, which is coprime with $N$ according to~\cref{clm:common_diff_relprime}.
    Therefore, $S(\Delta^{-1}A)$ contains all elements in $[N]$ except for three  consecutive integers $q+1,q+2,$ and $q+3.$ All three integers are nonzero (since $0 \in S(\Delta^{-1}A)$) and thus $2^n>q \ge 0$.
    Let us consider the binary representation $b_{n-1}b_{n-2}\ldots b_{0}$ of $q$,
    denoted as $q= \sum_{i=0}^{n-1} b_i 2^i.$
    Now, we define $B=\{ (-1)^{b_i+1} 2^i\},$ where $i$ ranges from $0$ to $n-1$.
    It follows that $S(B)=\{q+4-N, q+5-N,\ldots,q\}=S(\Delta^{-1}A)=\Delta^{-1} S(A).$

By~\cref{thr:GK23}, we can conclude that the set $\Delta^{-1}A$ is derived from a set $B$ by possibly switching the signs of some elements, and one can even conclude that $B=\Delta^{-1}A$, i.e. $A=\Delta B.$
\end{proof}

\section{Sumset-distinct sets modulo $N=2^n+2$}\label{sec:2^n+2}

\begin{proof}[Proof of~\cref{thm:2}]
    If $A$ contains $N/2$, then $A \setminus \{N/2\}$ must be sumset-distinct modulo $N/2.$ 
    Therefore, it is sufficient to prove the characterization assuming that $N/2 \not \in A.$
    
    Let $A$ be a sumset-distinct set modulo $N$ which does not contain $N/2.$
    Let $s_1$ and $s_2$ be the two missing elements in $S(A)$, which are the elements in $[N]$ that cannot be expressed as a sum of a subset of $A$.
    Suppose $A=\{a_1,\dots, a_n\}.$
    \begin{claim}       $s_1-s_2 \not \equiv N/2 \pmod N.$
    \end{claim}
    \begin{poc}
    
        Let $g_A(x)=\prod_i (1+x^{a_i})$.
    Then we have $g_A(\omega_N)+\omega_N^{s_1}+\omega_N^{s_2}= \sum_{i=0}^{N-1}\omega_N^i=0.$
    Since $N/2\not\in A$, $g_A(\omega_N)$ is nonzero. It follows that $\omega_N^{s_1}+\omega_N^{s_2}$ is also nonzero, which implies that $\omega_N^{s_2-s_1} \not= -1$. Therefore,  $s_1-s_2 \not \equiv N/2 \pmod N$, as claimed.    
    \end{poc}

    Since $2^{n}>2^{n-1}+1,$ $S(A)$ contains odd elements.
    This implies that $A$ contains at least one odd element.  By induction, we can conclude that $S(A)$ contains an equal number of odd and even sums. Consequently, $s_1$ and $s_2$ have different parities. Next, we will prove that $A$ contains exactly one odd element.

    \begin{claim}
        $A$ contains precisely one odd element. 
    \end{claim}
    \begin{poc}
        Assume that $A$ contains at least two odd numbers, with $a$ being one of them.
        Let $d=\gcd(a,N).$ 
        Since $S(A)=S(A \setminus \{a\})\cup (S(A \setminus \{a\})+a)$, every element modulo $d$ appears an even number of times.
        In particular, $s_1$ and $s_2$ have the same residue modulo $d.$
        Let $S_{i}(A)=\{s \in S(A) \mid s \equiv i \pmod d\}$ and let $S_{i}(A \setminus \{a\})$ be defined analogously.
        Since $s \in S(A \setminus \{a\})$ if and only if $s+a \in a+S(A \setminus \{a\}),$ for every $i \not \equiv s_1 \pmod d,$ $S_{i}(A \setminus \{a\})$ contains $\frac{N}{2d}$ elements, all of which have the same parity.
        Since $d=\gcd(a,N)$ and the two elements $s_1$ and $s_2$ have the same residue modulo $d$ but different parities,
        we can assume that $s_2 = s_1 - (2r-1)a \pmod N$, for some $r\in [N/2d]$. Then, 
        $$S_{s_1}(A \setminus \{a\})=\{ s_1 - 2ta \mid r> t\ge 1\} \cup \{ s_1 - (2t+1)a \mid N/2d > t \ge  r\} \pmod N,$$ 
        as $s \in S(A \setminus \{a\})$ if and only if $s+a \in a+S(A \setminus \{a\})$. The set $S_{s_1}(A \setminus \{a\})$ contains $\frac{N}{2d}-1$ elements, and since $s_1-s_2 \not= N/2 \pmod N,$  we have $r\ne (N/2d +1)/2$.
        This implies that the number of even and odd numbers in this set is not equal.
        Nevertheless, since $A \setminus \{a\}$ has at least one odd number by assumption, $S(A \setminus \{a\})$ contains the same number of odd as even numbers, which leads to  a contradiction.
    \end{poc}

    Denoting the single odd element of $A$ as $a$, we observe that $A \setminus \{a\}$ consists entirely of even numbers.
    As $A$ is sumset-distinct modulo $N$,
    the same applies to  $A \setminus \{a\}$. Hence, $(A \setminus \{a\})/2$ is sumset-distinct modulo $N/2$, and the conclusion follows.
\end{proof}

\section{Concluding remarks}
In this paper, we introduced a modular setting of the Erd\H{o}s distinct subset sums conjecture, and proved it asymptotically.
One may also investigate if a similar result holds for sets that are almost sumset-distinct modulo $N$, i.e.~sets having at least $(1-o(1))2^n$ different subset sums. 

We also characterise the sumset-distinct sets modulo $N=2^n+t$ of cardinality $n$, for $0 \le t \le 3.$ For larger $t$, we suggest the following problems. 
\begin{itemize}
    \item For every $t \ge 1,$ is the number of sumset-distinct sets modulo $N=2^n+t$ of cardinality $n$ bounded by $O_t(N^{2+v_2(t)})$?
    
    \item Given $t\ge 1$, let $n$ be sufficiently large. If $t$ is odd, is every sumset-distinct set modulo $N=2^n+t$ of cardinality $n$ a `perturbation' of the powers of $2$?
    
    \item If $t$ is even and $n$ sufficiently large, is it true that every sumset-distinct set modulo $N=2^n+t$ of cardinality $n$ either contains $N/2$, or contains exactly one odd number, or is a `perturbation' of the powers of $2$?
\end{itemize}

Here are two ways to construct perturbations of the powers of $2$.
\begin{itemize}
    \item $A=\{2^0+p_0, 2^1+p_1, \ldots, 2^{n-1}+p_{n-1}\}$, with $p_i\ge 0$, $p_{n-1}<t/2$ and $ p_{i+1} \ge \sum_{j=0}^{i} p_j$ for all $i\ge 0$,
    \item $A=\{a_0,a_1,\ldots,a_{i-1},2^i,2^{i+1},\ldots,2^{n-2}, 2^{n-1}\}$ with $v_2(a_j)=j$ for every $0 \le j \le i-1$.
\end{itemize}

We remark that the condition that $n$ is sufficiently large is necessary. For $N=2^5+4$, the set $\{1, 6, 11, 13, 15\}$ is sumset-distinct modulo $N$, but it is not of the above form.

In~\cref{sec:app1} it is additionally proven that many of the constructions of the first kind belong to different equivalence classes. 
It might be tempting to guess that for odd $t$ and $n$ large enough, every sumset-distinct set modulo $N=2^n+t$ is in the equivalence class of a set for which the sum of elements is bounded by $N$. This turns out to be false, as explained in~\cref{sec:app2}.

\section*{Acknowledgement}
The authors thank Rob Morris for helpful discussions at the early stage. The first author also thanks Tijs Buggenhout, Lei Yu, Noah Kravitz and Federico Glaudo for discussions.
\bibliographystyle{abbrv}
\bibliography{sumsets}

\section*{Appendix}\label{sec: appendix}

\appendix

\section{Perturbations of powers of $2$ that are sumset-distinct}\label{sec:app1}

We start with two families of sets that are sumset-distinct modulo $2^n+t$. Additionally, we can prove that many of these sets belong to distinct equivalence classes.

\begin{prop}\label{ex:constr2^n+2t+3}
    Let $r\ge 1$ be a fixed value, and set $t=2r+3$.
    Let us consider a set $A=\{2^0+p_0, 2^1+p_1, \ldots, 2^{n-1}+p_{n-1}\}$, which is a sumset-distinct sets modulo $2^n+t$. This holds for any choice of $p_i\ge 0$ that satisfies
    the condition $p_{i+1} \ge \sum_{j=0}^{i} p_j$ for all $i\ge 0$, and $r \ge p_{n-1}$.
    For sufficiently large $n$, each set of this form belongs to a distinct equivalence class.
\end{prop}

\begin{proof}
    Note that the given conditions imply that $t>2r \ge \sum_{j=0}^{n-1} p_j$.
    It is now sufficient to observe that the sum of all elements is bounded by $2^n-1+t$, and no two different subsets can have equal sums. This is a consequence of  the fact that $a_{i+1} > \sum_{j=0}^{i} a_j$ holds for every $i \ge 0$, where $a_j=2^j+p_j.$

    Next, we proceed to prove that two distinct sets of the given form, $A=\{a_1, a_2 \ldots, a_n\}=\{2^0+p_0, 2^1+p_1, \ldots, 2^{n-1}+p_{n-1}\}$ and $B=\{2^0+q_0, 2^1+q_1, \ldots, 2^{n-1}+q_{n-1}\}$, belong to different equivalence classes when $n\ge t> 2(r+1)$. 
    Assume, without loss of generality, that the number of $p_i$ values equal to zero is at least the number of $q_i$ values equal to zero.
    Let $j_0$ be the smallest index for which $q_{j_0}>0$.
    Since $r \ge p_{n-1}\ge \sum_{j=0}^{n-2} p_j$, there are at most $r+1$ occurrences of zero among the $p_i$ values.
    Since $n>2(r+1)$, we can conclude that more than half of the $p_i$ values(and similarly for $q_i$) are zeros.  Therefore, we have $j_0> \frac n2>r+1$.
    If $A$ and $B$ belong to the same equivalence class, there must exist a value $\lambda \in B$ such that $\{ \lambda a_1, \pm \lambda a_2, \ldots, \pm \lambda a_n\} $ is  equal to $B$.
    When $\lambda =2^s $ for some value of $1\le s\le j_0-1,$ it would imply that $2^{j_0} \in B$, leading to a contradiction.
    Therefore, we must have $\lambda=2^j+q_j$ for some $j \ge j_0.$
    Now, considering $2^{n-j-1}\lambda = 2^{n-1}+2^{n-j-1} q_j$, we need it to be equal to $\pm \left( 2^{n-1}+q_{n-1} \right) \pmod N$.
    The latter implies that $2^{n-j}\lambda = - \left( 2^{n}+2q_{n-1} \right) \pmod N,$ specifically equal to  $t-2q_{n-1}$.
    However, this value is an odd number that lies  strictly between $1$ and $2^{r+2}$, and therefore, it does not belong to $B.$
\end{proof}

For $t \in \{1,3,5,7\},$ it appears that the examples in \cref{ex:constr2^n+2t+3} generate all sumset-distinct sets modulo $2^n+t$. The number of examples, and therefore the number of equivalence classes in \cref{ex:constr2^n+2t+3}, is exponential in $t$ (related with \url{https://oeis.org/A131205}).
For larger odd values of $t$, there exist  alternative constructions as well.

\begin{examp}\label{ex:constr_pertubed_powersof2}
    A set $A=\{a_0,a_1,\ldots,a_{i-1},2^i,2^{i+1},\ldots,2^{n-2}, 2^{n-1}\}$ is sumset-distinct if $v_2(a_j)=j$ for every $0 \le j \le i-1$ and $N> 2^n-2^i +\sum_{j=0}^{i-1} a_j.$
\end{examp}

Note that these simple examples are all sumset-distinct in the original scenario and have a sum smaller than $N.$

\section{Sumset-disctinct sets could have large total sum}\label{sec:app2}
Our exact results for $N=2^n+t$, $0\le t\le 3$, and the plausible characterization for $t \in \{5,7\}$, suggest that for every odd $t>0$, every sumset-distinct set modulo $N=2^n+t$ could be equivalent to a set where the sum of elements is bounded by $N$. This turns out to be false for larger values of $t.$

\begin{prop}
    For $n\ge 8$ and $N=2^n+9$, there exists an equivalence class of 
    sumset-distinct sets modulo $N=2^n+9$
    that does not have a representative where the sum of the elements is smaller than $N.$
\end{prop}

\begin{proof}
    We will prove the validity of the statement for the sets $\{3,6,12,24,2^4,2^5,\ldots,2^{n-2},2^{n-1}\} \sim \{1,2,4,8,\ldots, 2^{n-5},5\cdot 2^{n-4}+3, 3\cdot 2^{n-3}+3,2^{n-2}+3,2^{n-1}+3  \}\sim \{3,4,6,8,16,\ldots,2^{n-3},2^{n-2}+3,2^{n-1}+3\}.$ 
    First, we will prove that the sets in the class are sumset-distinct modulo $N=2^n+9$.

    \begin{claim}
        The set $A=\{3,6,12,24,2^4,2^5,\ldots,2^{n-2},2^{n-1}\}$ is sumset-distinct modulo $N=2^n+9$.
    \end{claim}
    \begin{poc}
    Note that $s(A)=2^n+29.$
    The numbers have different $2$-adic valuations, thus the two subsets with equal sums should have sum $x$ and $2^n+9+x$.
    Since $x \le 10$, it must be a multiple of $3$.
    The elements not involved, summing to $20-2x$, should be multiples of $3$, or equal to $16$.
    However, since $20 \not \equiv 0,16 \pmod 3$, this is impossible.
    \end{poc}
            
    Next, we prove that no representative of this equivalence class has a sumset(sum of elements) smaller than $N.$

    \begin{claim}
        The set $A=\{1,2,4,8,\ldots, 2^{n-5},5\cdot 2^{n-4}+3, 3\cdot 2^{n-3}+3,2^{n-2}+3,2^{n-1}+3 \}$ has no representative whose elements sum to less than $N.$
    \end{claim}
    \begin{poc}
        Recall $\res(x)=\min\{ x \pmod N, N-( x\pmod N) \}$, where $x \pmod N$.
        We need to prove that for every $0< \lambda < \frac{N}{2}$, we have $\sum_{i=0}^{n-5} \res(2^i \lambda) + \sum_{i=1}^{4} \res\left( \frac{3}{2^i} \lambda\right ) \ge N.$
        According to ~\cref{prop:no_el_inN/3..}, we only need to consider values of $\lambda$ for which none of the elements belong to the interval $\left[ \frac N3, 2^{n-1}-1 \right].$
        For $\lambda \ge 17,$ we know that $\res( \lambda 2^i ) \ge \frac N3$ for some $0 \le i \le n-5.$
        Therefore, we can focus on the cases where either $ 2^{n-1} \le \lambda 2^i \le 2^{n-1} +9$ or $\lambda \le 17.$
        
        \textbf{Case 1: $1 \le \lambda \le 16$} These cases can be verified using a computer.\\ See \url{https://github.com/StijnCambie/ErdosSumSet/blob/main/lambda_till16.sagews}.
        
        \textbf{case 2: $\lambda \ge 17$, and $\lambda<\frac{N}{3}$ is even, or $ \lambda < \frac{N}{6}$}
       Note that there exists some $i\le n-6$ such that $2^i \lambda<\frac{N}{2}<2^{i+1} \lambda.$
        Now consider $\sum_{j=0}^{i+1} \res\left( 2^j \lambda \right) = (2^{i+1}-1)\lambda + (N-2^{i+1} \lambda) = N-\lambda.$
        If $\lambda$ is even, then $\res\left( \frac 32 \lambda \right) = \frac 32 \lambda $ and we are done.
        If $ \lambda < \frac{N}{6}$ is odd, then $\res\left( \frac 32 \lambda \right) = \frac{N}{2}-\frac 32 \lambda > \frac 32 \lambda $ as well.
        
        \textbf{Case 3: $\lambda$ is among a few remaining possibilities}
        
       These remaining cases can be also verified using a computer.\\ See \url{https://github.com/StijnCambie/ErdosSumSet/blob/main/largelambda.sagews}.
    \end{poc}
\end{proof}

\end{document}